\documentclass{article}

\usepackage{amsfonts}
\usepackage[english]{babel}
\usepackage{latexsym}
\usepackage[latin1]{inputenc}
\usepackage{amssymb}
\usepackage{amsmath}
\usepackage{bbm}
\usepackage{amsthm}
\usepackage{graphicx}

\newtheorem{theorem}{Theorem}[section]
\newtheorem{definition}[theorem]{Definition}
\newtheorem{lemma}[theorem]{Lemma}

\newcommand{\pd}[2]{\frac{\partial #1}{\partial #2}}

\title{Convergence of the Crank-Nicolson-Galerkin finite element method for a class of nonlocal parabolic systems with
moving boundaries}

\author{Rui M.P. Almeida \thanks{Department of Mathematics, Faculty of Science,
University of Beira Interior, http://www.cmatubi.ubi.pt, email: ralmeida@ubi.pt}
\and Jos\'{e} C.M. Duque \thanks{Department of Mathematics, Faculty of Science,
University of Beira Interior, http://www.cmatubi.ubi.pt, email: jduque@ubi.pt}
\and Jorge Ferreira \thanks{Federal University Rural of Pernambuco-UFRPE-UAG, Center for Mathematics and Fundamental Applications, Faculty of Science, University of Lisbon,http://cmaf.ptmat.fc.ul.pt, email:ferreirajorge2012@gmail.com }
\and Rui J. Robalo \thanks{Department of Mathematics, Faculty of Science,
University of Beira Interior, http://www.cmatubi.ubi.pt, email: rrobalo@ubi.pt}}
\date{\today}

\begin{document}
\maketitle

\begin{abstract}
The aim of this paper is to establish the convergence and error bounds to the fully discrete solution for a
class of nonlinear systems of reaction-diffusion nonlocal type with moving boundaries, using a linearized
Crank-Nicolson-Galekin finite element method with polynomial approximations of any degree.

A coordinate transformation which fixes the boundaries is used.
Some numerical tests to compare our Matlab code
with some existence moving finite elements methods are investigated.

\textbf{keywords}: nonlinear parabolic system, nonlocal diffusion term, reaction-diffusion, convergence, numerical
 simulation, Crank-Nicolson, finite element method.

\end{abstract}

\section{Introduction}
\label{intro}

In this work, we study parabolic systems with nonlocal nonlinearity of the
following type
\begin{equation}  \label{probi}
\left\{
\begin{array}{l}
\displaystyle \frac{\partial u_i}{\partial t}-a_i\left( \int_{\Omega
_{t}}u_1(x,t)dx,\dots,\int_{\Omega _{t}}u_{ne}(x,t)dx\right) \frac{%
\partial^2u_i}{\partial x^2}=f_i\left(x,t\right)\,, \quad (x,t)\in Q_{t} \\
\displaystyle u_i\left( \alpha (t),t\right) =u_i\left( \beta
(t),t\right)=0\,,\quad t>0 \\
\displaystyle u_i(x,0)=u_{i0}(x)\,, \quad x\in \Omega _{0}=]\alpha (0),\beta
(0)[,\quad i=1,\dots,ne \\
\end{array}
\right. \,
\end{equation}
where $Q_{t}$ is a bounded non-cylindrical domain defined by
\begin{equation*}
Q_{t}=\left\{ (x,t)\in \mathbbm{R}^{2}:\:\alpha(t)<x<\beta(t),\:\: \text{for
all } 0<t<T\right\}\,.
\end{equation*}

Problem (\ref{probi}) is nonlocal in the sense that the diffusion
coefficient is determined by a global quantity, that is, $a$ depends on the
whole population in the area and it arises in a large class of real models.
For example, in biology, where the solution $u$ could describe the density
of a population subject to spreading; or in physics, where $u$ could represent
the temperature, considering that the measurements are an average in the
neighbourhood \cite{chi00}.

This class of problems with nonlocal coefficient in an open bounded
cylindrical domain was initially studied by Chipot and Lovat  in \cite{CL97}
, where they proved the existence and uniqueness of weak solutions. In
recent years nonlinear parabolic equations with nonlocal diffusion terms
have been extensively studied \cite
{CL99,AK00,CM01,CC03a,CS03,CVC03,CMF04,ZC05}, especially in relation to
questions of existence, uniqueness and asymptotic behaviour.

If we want to model interactions then we need to use a system. Raposo et al.
\cite{RSVPS08}, in 2008, studied the existence, uniqueness and exponential
decay of solutions for reaction-diffusion coupled systems of the form
\begin{equation*}
\left\{
\begin{array}{lrl}
u_{t}-a(l(u))\Delta u+f(u-v)=\alpha (u-v) & \text{ in } & \Omega \times
]0,T], \\
v_{t}-a(l(v))\Delta v-f(u-v)=\alpha (v-u) & \text{ in } & \Omega \times
]0,T], \\
&  &
\end{array}%
\right.
\end{equation*}%
with $a(\cdot )>0$, $l$ a continuous linear form, $f$ a Lipschitz-continuous
function and $\alpha $ a positive parameter. Recently, Duque et al. \cite%
{DAAFppa} considered nonlinear systems of parabolic equations with a more
general nonlocal diffusion term working on two linear forms $l_{1}$ and $%
l_{2}$:
\begin{equation}
\left\{
\begin{array}{lrl}
u_{t}-a_{1}(l_{1}(u),l_{2}(v))\Delta u+\lambda _{1}|u|^{p-2}u=f_{1}(x,t) &
\text{ in } & \Omega \times ]0,T], \\
v_{t}-a_{2}(l_{1}(u),l_{2}(v))\Delta v+\lambda _{2}|v|^{p-2}v=f_{2}(x,t) &
\text{ in } & \Omega \times ]0,T]. \\
&  &
\end{array}%
\right.   \label{sisDAAF}
\end{equation}%
They gave important results on polynomial and exponential decay, vanishing
of the solutions in finite time and localization properties such as waiting
time effect.

Moving boundary problems occur in many physical applications involving
diffusion, such as in heat transfer where a phase transition occurs, in
moisture transport such as swelling grains or polymers, and in deformable
porous media problems where the solid displacement is governed by diffusion,
(see for example, \cite{fer97,fer99,fer05,bri07}). Cavalcanti et al \cite%
{fer03}  worked with a time-dependent function $a=a\left( t,\int_{\Omega
_{t}}\left\vert \nabla u(x,t)\right\vert ^{2}dx\right) $ to establish the
solvability and exponential energy decay of the solution for a model given
by a hyperbolic-parabolic equation in an open bounded subset of $\mathbbm{R}%
^{n}$, with moving boundary. Santos et al. \cite{san05} established the
exponential energy decay of the solutions for nonlinear coupled systems for
beam equations with memory in noncylindrical domains. Recently, Robalo et
al. \cite{RACFppa} proved the existence and uniqueness of weak and strong
global in time solutions and gave conditions, on the data, for these
solutions to have the exponential decay property. The analysis and numerical
simulation of such problems presents other challenges. In \cite{AK00},
Ackleh and Ke propose a finite difference scheme to approximate the
solutions and to study their long time behavior. The authors also made
numerical simulations,  using an implicit finite difference scheme in one
dimension \cite{RSVPS08} and the finite volume discretization  in two space
dimensions \cite{EGHM02}. Bendahmane and Sepulveda \cite{BS09} in 2009
investigated the propagation of an epidemic disease modeled by a system of
three PDE, where the $i$th equation is of the type
\begin{equation*}
(u_{i})_{t}-a_{i}\left( \int_{\Omega }u_{i}dx\right) \Delta
u_{i}=f_{i}\left( u_{1},u_{2},u_{3}\right) ,
\end{equation*}%
in a physical domain $\Omega \subset \mathbbm{R}^{n}$, $(n=1,2,3)$.
They established the existence of solutions to finite volume scheme and its
convergence to the weak solution of the PDE. In \cite{DAAFppb} the authors
proved the optimal order of convergence for a linearized Euler-Galerkin
finite element method to  problem (\ref{sisDAAF}) and presented some
numerical results. Almeida et al., in \cite{ADFRpp}, established the
convergence and error bounds of the fully discrete solutions for a class of
nonlinear equations of reaction-diffusion nonlocal type with moving
boundaries, using a linearized Crank-Nicolson-Galerkin finite element method
with polynomial approximations of any degree. In \cite{RACFppb}, Robalo et
al. proved the existence and uniqueness of a strong regular solution for a
certain class of a nonlinear coupled system of reaction-diffusion equations
on a bounded domain with moving boundary. The exponential decay of the
energy of the solutions, under the same assumptions, was also proved. In
addition, they obtained approximate numerical solutions for systems of this
type with a Matlab code based on the Moving Finite Element Method (MFEM)
with high degree local approximations.

This paper is concerned with the proof of the convergence of a total
discrete solution using the Crank-Nicolson-Galekin finite element method. To the best of our knowledge, these results are new for nonlocal
reaction-diffusion systems with moving boundaries.

The paper is organized as follows. In Section 2, we formulate the problem and the hypotheses on the
data. In Section 3, we define and prove the convergence of the semidiscrete
solution. Section 4 is devoted to the proof of the convergence to a fully
discrete solution. In Section 5, we obtain approximate numerical solutions
for some examples. To finalize this study, in Section 6, we draw some
conclusions.

\section{Statement of the problem}

In what follows, we study the convergence of a linearized
Crank-Nicolson-Galerkin finite element method to the solutions of the
one-dimensional Dirichlet problem with two moving boundaries, defined by
\begin{equation}  \label{prob0}
\left\{
\begin{array}{l}
\displaystyle \frac{\partial u_i}{\partial t}-a_i\left( \int_{\Omega
_{t}}u_1(x,t)dx,\dots,\int_{\Omega _{t}}u_{ne}(x,t)dx\right) \frac{%
\partial^2u_i}{\partial x^2}=f_i\left(x,t\right), \,(x,t)\in Q_{t} \\
\displaystyle u_i\left( \alpha (t),t\right) =u_i\left( \beta
(t),t\right)=0\,,\quad t>0 \\
\displaystyle u_i(x,0)=u_{i0}(x)\,, \quad x\in \Omega _{0}=]\alpha (0),\beta
(0)[,\quad i=1,\dots,ne \\
\end{array}
\right. \,
\end{equation}

where
\begin{equation*}
Q_{t}=\left\{ (x,t)\in\mathbbm{R}^{2}:\:\alpha(t)<x<\beta(t),\:\: \text{for
all } 0<t<T\right\}\,
\end{equation*}
is a bounded non-cylindrical domain, $T$ is an arbitrary positive real
number and $a_i$ denotes a positive real function. The lateral boundary of $%
Q_ {t} $ is given by $\Sigma_{t}=\bigcup_{0\leq t<T} \left(
\left\{\alpha(t),\beta(t)\right\}\times \{t\}\right)$. Moreover, we assume
that $\alpha^{\prime }(t)< 0$ and $\beta^{\prime }(t)> 0$, for all $t\in
[0,T]$. Note that the hypotheses $\alpha^{\prime }(t)< 0$ and $\beta^{\prime
}(t)> 0$ imply that $Q_{t}$ is increasing, in the sense that if $t_{2}>t_{1}$%
, then the projection of $\left[\alpha(t_{1}),\beta(t_{1})\right]$ onto the
subspace $t=0$ is contained in the projection of $\left[\alpha(t_{2}),%
\beta(t_{2})\right]$ onto the same subspace. This also means that the real
function $\gamma(t)=\beta(t)-\alpha(t)$ is increasing on $0\leq t<T$.

In \cite{RACFppb} Robalo et al. established the existence, uniqueness and
asymptotic behaviour of strong regular solutions for these problems using a
coordinate transformation which fixes the boundaries. They used the fact
that, when $(x,t)$ varies in $Q_{t}$, the point $(y,t)$ of $\mathbbm{R}^{2}$%
, with $y=(x-\alpha(t))/\gamma(t)$, varies in the cylinder $Q=]0,1[
\times]0,T[$. Thus, the function $\tau :Q_{t} \longrightarrow Q$ given by $%
\tau(x,t)=(y,t)$, is of class $\mathcal{C}^{2}$. The inverse $\tau^{-1}$ is
also of class $\mathcal{C}^{2}$. The change of variable $v(y,t)=u(x,t)$ and $%
g(y,t)=f(x,t)$ with $x=\alpha(t)+\gamma(t)\,y$ transforms problem (\ref%
{prob0}) into problem (\ref{prob1}), given by
\begin{equation}  \label{prob1}
\left\{
\begin{array}{l}
\displaystyle \frac{\partial v_i}{\partial t}-a_i\left(
l(v_1),\dots,l(v_{ne})\right)b_2(t) \frac{\partial^2v_i}{\partial y^2}%
-b_1(y,t)\frac{\partial v_i}{\partial y}=g_i\left(y,t\right)\,, \quad
(y,t)\in Q \\
\displaystyle v_i\left(0,t\right) =v_i\left(1,t\right)=0\,,\quad t>0 \\
\displaystyle v_i(y,0)=v_{i0}(y)\,, \quad y\in \Omega=]0,1[,\quad
i=1,\dots,ne \\
\end{array}
\right. \,
\end{equation}
where $l(v)=\gamma(t)\int_0^1 v(y,t)\ dy$, $g_i(y,t)=f_i(\alpha+\gamma\,y,t)$
and $v_{i0}(y)=u_{i0}(\alpha(0)+\gamma(0)\,y)$. The coefficients $b_{1}(y,t)$
and $b_{2}(t)$ are defined by
\begin{equation*}
b_{1}(y,t)=\frac{\alpha^{\prime }(t)+\gamma^{\prime }(t)y}{\gamma(t)}\quad
\text{and} \quad b_{2}(t)=\frac{1}{\left(\gamma(t)\right)^{2}} \,.
\end{equation*}

Since we need the existence and uniqueness of a strong solution in $Q_{t}$,
we consider the same hypotheses as in \cite{RACFppb}, namely:
\begin{equation*}
\begin{array}{l}
(H1)\qquad \alpha ,\,\beta \in \mathcal{C}^{2}\left( [0,T]\right) \text{ and
 }0<\gamma _{0}<\gamma (t)<\gamma _{1}<\infty \,,\;\text{for all }t\in
\lbrack 0,T] \\
(H2)\qquad \alpha ^{\prime },\,\beta ^{\prime }\in L_{1}\left( ]0,T[\right)
\cap L_{2}\left( ]0,T[\right)  \\
(H3)\qquad u_{i0}\in H_{0}^{1}\left( \Omega _{0}\right) \,,\quad \Omega
_{0}=]\alpha (0),\beta (0)[,\quad i=1,\dots ,ne, \\
(H4)\qquad f_{i}\in L_{2}\left( 0,T;L_{2}\left( \Omega _{t}\right) \right)
\cap L_{1}\left( 0,T;L_{2}\left( \Omega _{t}\right) \right) \,, \\
\hspace{1.9cm}\Omega _{t}=]\alpha (t),\beta (t)[,\,i=1,\dots ,ne, \\
(H5)\qquad a_{i}:\mathbbm{R}^{ne}\longrightarrow \mathbbm{R}^{+}\text{
is Lipschitz-continuous } \\
\hspace{1.9cm}\text{with }0<m_{a}\leq a_{i}(s)\leq M_{a}\,,\;\;\text{for all
}s\in \mathbbm{R},\quad i=1,\dots ,ne. \\
\end{array}%
\end{equation*}%
Let $\Omega =]0,1[$. The definition of a weak solution is as follows.

\begin{definition}[Weak solution]
\label{fraca} We say that the function $\mathbf{v}=(v_{1},\dots ,v_{ne})$ is
a weak solution of problem (\ref{prob1}) if, for each $i\in \{1,\dots ,ne\}$,
\begin{equation}
v_{i}\in L_{\infty }(0,T;H_{0}^{1}(\Omega )\cap H^{2}(\Omega )),\frac{%
\partial v_{i}}{\partial t}\in L_{2}(0,T;L_{2}(\Omega )),  \label{regx}
\end{equation}
the following equality in $D^{\prime }(0,T)$ is valid for all $w_{i}\in H_{0}^{1}(\Omega ),$ and $t\in ]0,T[$,
\begin{equation}
\int_{0}^{1}\frac{\partial v_{i}}{\partial t}w_{i}dy+a_{i}(l(v_{1}),\dots
,l(v_{ne}))b_{2}\int_{0}^{1}\frac{\partial v_{i}}{\partial y}\frac{\partial
w_{i}}{\partial y}dy-\int_{0}^{1}b_1\frac{\partial v_{i}}{\partial y}w_{i}dy=\int_{0}^{1}g_{i}w_{i}dy  \label{fracav}
\end{equation}
 and
\begin{equation}
v_{i}(x,0)=v_{i0}(x),\quad x\in \Omega   \label{condiuv}
\end{equation}
\end{definition}

Henceforth, we assume that $\mathbf{v}$ has the regularity needed to perform
all the calculations which follow.

\section{Semidiscrete solution}

We denote the usual $L_{2}$ norm in $\Omega $ by $\Vert .\Vert $ and the
norm in $H^{k}(\Omega )$ by $\Vert .\Vert _{H^{k}}$.\newline
Let $\mathcal{T}_{h}$ denote a partition of $\Omega $ into disjoint
intervals $T_{i}$, $i=1,\dots ,nt$ such that $h=\max \{diam(T_{i}),i=1,\dots
,nt\}$. Now let $S_{h}^{k}$ denote the continuous functions on the closure $%
\bar{\Omega}$ of $\Omega $ which are polynomials of degree $k$ in each
interval of $\mathcal{T}_{h}$ and which vanish on $\partial \Omega $, that
is,
\begin{equation*}
S_{h}^{k}=\{W\in C_{0}^{0}(\bar{\Omega})|{W}_{|{T_{i}}}\text{ is a
polynomial of degree $k$ for all }T_{i}\in \mathcal{T}_{h}\}.
\end{equation*}%
If $\{\varphi _{j}\}_{j=1}^{np}$ is a basis for $S_{h}^{k}$, then we can
represent each $W\in S_{h}^{k}$ as
\begin{equation*}
W=\sum_{j=1}^{np}w_{j}\varphi _{j}.
\end{equation*}%
Given a smooth function $u$ on $\Omega ,$ which vanishes on $\partial \Omega
$, we may define its interpolant, denoted by $I_{h}u$, as the function of $%
S_{h}^{k}$ which coincides with $u$ at the points $\{P_{j}\}_{j=1}^{np}$,
that is,
\begin{equation*}
I_{h}u=\sum_{j=1}^{np}u(P_{j})\varphi _{j}.
\end{equation*}

\begin{lemma}[\protect\cite{Tho06}]
\label{errint} If $u\in H^{k+1}(\Omega)\cap H_0^1(\Omega)$, then
\begin{equation*}
\|I_h u-u\|+h\|\nabla(I_h u-u)\|\leq Ch^{k+1}\|u\|_{H^{k+1}}.
\end{equation*}
\end{lemma}

\begin{definition}[\protect\cite{Tho06} Ritz projection]
A function $\tilde u\in S_h^k$ is said to be the Ritz projection of $u\in
H_0^1(\Omega)$ onto $S_h^k$ if it satisfies
\begin{equation*}
\int_{\Omega}\nabla \tilde u\cdot\nabla W\ dy=\int_{\Omega}\nabla
u\cdot\nabla W\ dy,\quad \text{ for all } W\in S_h^k.
\end{equation*}
\end{definition}

\begin{lemma}[\protect\cite{Tho06}]
\label{errproj} If $u\in H^{k+1}(\Omega )\cap H_{0}^{1}(\Omega )$, then
\begin{equation*}
\Vert \tilde{u}-u\Vert +h\Vert \nabla (\tilde{u}-u)\Vert \leq Ch^{k+1}\Vert
u\Vert _{H^{k+1}},
\end{equation*}%
where $C$ does not depend on $h$ or $k$.
\end{lemma}

The semidiscrete problem, based on Definition \ref{fraca}, consists in
finding\linebreak $\mathbf{V}=(V_1,\dots,V_{ne})$ belonging to $(S_h^k)^{ne}$, for $
t\geq 0$, such that for all\linebreak $\mathbf{W}=(W_1,\dots,W_{ne})\in (S_h^k)^{ne}$
and $t\in]0,T[$:
\begin{equation}  \label{probsd}
\left\{%
\begin{array}{l}
\displaystyle\int_{0}^{1}\frac{\partial V_{i}}{\partial t}%
W_idy+a(l(V_1),\dots,l(V_{ne}))b_{2}\int_{0}^{1}\frac{\partial V_i}{
\partial y}\frac{\partial W_i}{\partial y}\ dy-\int_{0}^{1}b_{1}\frac{
\partial V_i}{\partial y}W_i\ dy \\[0.7cm]
\displaystyle\hfill\hspace{7cm}=\int_{0}^{1}g_iW_i\ dy \\
V_i(y,0)=I_hv_{i0},\quad i=1,\dots,ne \\
\end{array}%
\right..
\end{equation}

\begin{theorem}
\label{conv_h} If $\mathbf{v}$ is the solution of problem (\ref{prob1}) and $%
\mathbf{V}$ is the solution of (\ref{probsd}), then
\begin{equation*}
\Vert V_{i}-v_{i}\Vert \leq Ch^{k+1},\quad t\in ]0,T],\quad i=1,\dots ,ne
\end{equation*}%
where $C$ does not depend on $h$, $k$ or $i$.
\end{theorem}

\begin{proof}
Let $e_i=V_i-v_i$ be written as
\begin{equation*}
e_i(y,t)=(V_i(y,t)-\tilde{V}_i(y,t))+(\tilde{V}_i
(y,t)-v_i(y,t))=\theta_i (y,t)+\rho_i (y,t),
\end{equation*}
with  $\tilde{V}_i^{(h)}(y,t)\in S_{h}^{k}$ being the Ritz projection of $v_i$. Then
\begin{equation*}
\left\Vert e_i(y,t)\right\Vert\leq \left\Vert \theta_i
(y,t)\right\Vert+\left\Vert \rho_i (y,t)\right\Vert
\end{equation*}
and, by lemma \ref{errproj}, it follows that
\begin{equation*}
\left\Vert \rho_i (y,t)\right\Vert\leq Ch^{k+1}\left\Vert
v_i\right\Vert _{H^{k+1}}
\end{equation*}
Next, we determine an upper limit for $\left\Vert \theta_i (y,t)\right\Vert$. Let
\begin{equation*}
a_{i}^{(h)}=a_i(l(V_1),\dots,l(V_{ne})).
\end{equation*}
Then, for every $i\in\{1,\dots,ne\}$, we have that\\

$\displaystyle\int_{0}^{1}\pd{\theta_i}{t} W_i dy+a_{i}^{(h)}b_{2}\int_{0}^{1}\frac{\partial \theta_i
}{\partial y}\frac{\partial W_i}{\partial y}dy-\int_{0}^{1}b_{1}\frac{
\partial \theta_i }{\partial y}W_i dy $
\begin{eqnarray*}
&=&\int_{0}^{1}\pd{V_i}{t}W_idy+a_{i}^{(h)}b_{2}\int_{0}^{1}\frac{\partial V_i
}{\partial y}\frac{\partial W_i}{\partial y}dy-\int_{0}^{1}b_{1}\frac{
\partial V_i}{\partial y}W_i dy \\
&&-\int_{0}^{1}\pd{\tilde{V}_{i}}{t}W_i dy-a_{i}^{(h)}b_{2}\int_{0}^{1}
\frac{\partial \tilde{V}_i}{\partial y}\frac{\partial W_i}{
\partial y}dy+\int_{0}^{1}b_{1}\frac{\partial \tilde{V}_i}{
\partial y}W_i dy \\
&=&\int_{0}^{1}g_iW_{i}dy-\int_{0}^{1}\pd{v_{i}}{t}W_i dy-a_ib_{2}\int_{0}^{1}\frac{
\partial \tilde{V}_i}{\partial y}\frac{\partial W_i}{
\partial y}dy+\int_{0}^{1}b_{1}\frac{\partial v_i}{\partial y}W_i dy \\
&&+(a_i-a_{i}^{(h)})b_{2}\int_{0}^{1}\frac{\partial \tilde{V}_i}{
\partial y}\frac{\partial W_i}{\partial y}dy+\int_{0}^{1}b_{1}(\frac{\partial \tilde{V}_i}{\partial y}-
\frac{\partial v_i}{\partial y})W_i dy\\
&&+\int_{0}^{1}(\pd{v_{i}}{t}-\pd{\tilde{V}_{i}}{t})W_idy \\
&=&(a_i-a_{i}^{(h)})b_{2}\int_{0}^{1}\frac{\partial \tilde{V}_i}{
\partial y}\frac{\partial W_i}{\partial y}dy+\int_{0}^{1}b_{1}(\frac{
\partial \tilde{V}_i}{\partial y}-\frac{\partial v_i}{\partial y}
)W_i dy\\
&&+\int_{0}^{1}(\pd{v_{i}}{t}-\pd{\tilde{V}_{i}}{t})W_i dy.
\end{eqnarray*}
If we consider $W_i=\theta_i$, then\\

$\displaystyle
\int_{0}^{1}\pd{\theta_{i}}{t}\theta_i dy+a_{i}^{(h)}b_{2}\int_{0}^{1}\left( \frac{\partial \theta_i }{\partial y}\right) ^{2}dy-\int_{0}^{1}b_{1}\frac{\partial \theta_i }{\partial y}\theta_i dy$\\

\begin{flushright}
$\displaystyle
=(a_i-a_{i}^{(h)})b_{2}\int_{0}^{1}\frac{\partial \tilde{V}_i}{\partial y}\frac{\partial \theta_i }
{\partial y}dy+\int_{0}^{1}b_{1}\frac{\partial \rho_i }{\partial y}\theta_i
dy-\int_{0}^{1}\pd{\rho_{i}}{t}\theta_i dy.$
\end{flushright}
Integrating by parts the third term on the left side and the second term on the right side of the above equation, we obtain\\

$\displaystyle
\int_{0}^{1}\frac{1}{2}\frac{d}{dt}\theta_i
^{2}dy+a_{i}^{(h)}b_{2}\int_{0}^{1}\left( \frac{\partial \theta_i }{\partial y}
\right) ^{2}dy+\frac{\gamma ^{^{\prime }}(t)}{2\gamma (t)}\int_{0}^{1}\theta_i
^{2}dy$\\

\begin{flushright}
$\displaystyle
=(a_i-a_{i}^{(h)})b_{2}\int_{0}^{1}\frac{\partial \tilde{V}_i
}{\partial y}\frac{\partial \theta_i }{\partial y}dy-\int_{0}^{1}\pd{\rho_{i}}{t}\theta_i dy-\frac{\gamma ^{^{\prime }}(t)}{\gamma (t)}%
\int_{0}^{1}\rho_i \theta_i dy-\int_{0}^{1}b_{1}\rho_i \frac{\partial \theta_i }{
\partial y}dy.$
\end{flushright}
Taking the absolute value of the right side of this equation, ignoring the third term on the left side and considering the lower limits of $a$ and $b_{i}$, it follows that\\
$\displaystyle\frac12\frac{d}{dt}\|\theta_i\|^2+\frac{m_a}{\gamma_1^2}\left\|\pd{\theta_i}{y}\right\|^2$\\

$\displaystyle
\leq\left\vert a_i-a_{i}^{(h)}\right\vert \frac{1}{\gamma _{0}^{2}}\int_{0}^{1}\left\vert \frac{\partial\tilde{V}_i}
{\partial y}\right\vert \left\vert \frac{\partial \theta_i }{\partial y}\right\vert dy+\int_{0}^{1}\left\vert \pd{\rho_{i}}{t}\right\vert
\left\vert \theta_i \right\vert dy+\frac{\gamma _{\max }^{^{\prime }}}{\gamma _{0}}\int_{0}^{1}\left\vert
\rho_i \right\vert \left\vert \theta_i \right\vert dy$\\

$\displaystyle\quad +\frac{\alpha _{\max
}^{\prime }+\gamma _{\max }^{^{\prime }}}{\gamma _{0}}\int_{0}^{1}\left\vert
\rho_i \right\vert \left\vert \frac{\partial \theta_i }{\partial y}\right\vert dy$\\

$\displaystyle\leq C_{1}\left\vert a_i-a_{i}^{(h)}\right\vert^2+ \frac{m_a}{2\gamma _{1}^{2}}
\int_{0}^{1}\left\vert \frac{\partial \theta_i }{\partial y}\right\vert ^{2}dy+
\frac{1}{2}\int_{0}^{1}\left\vert \pd{\rho_{i}}{t}\right\vert ^{2}dy+\frac{1}{2}
\int_{0}^{1}\left\vert \theta_i \right\vert ^{2}dy$\\

$\displaystyle\quad +\frac{\gamma _{\max }^{^{\prime }}}{2\gamma _{0}}\int_{0}^{1}\left\vert
\rho_i \right\vert ^{2}dy+\frac{\gamma _{\max }^{^{\prime }}}{2\gamma _{0}}
\int_{0}^{1}\left\vert \theta_i \right\vert ^{2}dy+C_{2}\int_{0}^{1}\left\vert
\rho_i \right\vert ^{2}dy+\frac{m_{a}}{2\gamma _{1}^{2}}\int_{0}^{1}\left\vert
\frac{\partial \theta_i }{\partial y}\right\vert ^{2}dy.$\\
Then, by (H5) we have that
\begin{eqnarray*}
\frac{1}{2}\frac{d}{dt}\left\Vert \theta_i \right\Vert ^{2} &\leq
&C_{3}\sum_{j=1}^{ne}\left\Vert \rho_j \right\Vert^{2}+C_{4}\sum_{j=1}^{ne}\left\Vert \theta_j
\right\Vert ^{2}+\frac{1}{2}\left\Vert \pd{\rho_{i}}{t}\right\Vert
^{2}+\frac{1}{2}\left\Vert \theta_i \right\Vert^{2}
\\
&&+\frac{\gamma _{\max }^{^{\prime }}}{2\gamma _{0}}\left\Vert \rho_i
\right\Vert^{2}+\frac{\gamma _{\max }^{^{\prime }}}{2\gamma
_{0}}\left\Vert \theta_i \right\Vert^{2}+C_{2}\left\Vert \rho_i
\right\Vert^{2} \\
&\leq &C\sum_{j=1}^{ne}\left\Vert \theta_j \right\Vert^{2}+C\sum_{j=1}^{ne}\left\Vert \rho_j
\right\Vert^{2}+\frac{1}{2}\left\Vert \pd{\rho_{i}}{t}\right\Vert^{2}.
\end{eqnarray*}
and hence, we obtain\\
$$ \frac{d}{dt}\left(\sum_{i=1}^{ne}\left\Vert \theta_i \right\Vert ^{2}\right)\leq C\sum_{i=1}^{ne}\left\Vert \theta_i \right\Vert^{2}+C\sum_{i=1}^{ne}\left\Vert \rho_i
\right\Vert^{2}+\sum_{i=1}^{ne}\left\Vert \pd{\rho_{i}}{t}\right\Vert^{2}.$$
Applying Gronwall's Theorem, we arrive at the inequality
\begin{equation*}
\sum_{i=1}^{ne}\left\Vert \theta_i \right\Vert^{2}\leq C\sum_{i=1}^{ne}\left\Vert \theta_i
(y,0)\right\Vert^{2}+C\sum_{i=1}^{ne}\int_{0}^{T}\left\Vert \rho_i \right\Vert^{2}+\left\Vert \pd{\rho_{i}}{t}\right\Vert^{2}dt.
\end{equation*}
By the hypothesis of the theorem, we have, for every $i\in{1,\dots,ne}$,
\begin{eqnarray*}
\left\Vert \theta_i (y,0)\right\Vert^{2} &\leq &\left\Vert
e_i(y,0)\right\Vert^{2}=\left\Vert V_i(y,0)-v_{i0}\right\Vert
^{2}\leq Ch^{2(k+1)}\|v_{i0}\|_{H^{k+1}}^{2}, \\
\int_{0}^{T}\left\Vert \rho_i \right\Vert^{2} \ dt&\leq
&CTh^{2(k+1)}\left\Vert v_i\right\Vert _{H^{k+1}}^{2},\\
\int_{0}^{T}\left\Vert \pd{\rho_{i}}{t} \right\Vert^{2}\ dt &\leq
&CTh^{2(k+1)}\left\|\pd{v_{i}}{t}\right\|_{H^{k+1}}^2
\end{eqnarray*}
and so
\begin{equation*}
\displaystyle\sum_{i=1}^{ne}\left\Vert \theta_i \right\Vert^{2}\leq C\left( \sum_{i=1}^{ne}\left\Vert
v_{i0}\right\Vert _{H^{k+1}}^{2}+\sum_{i=1}^{ne}\left\Vert v_i\right\Vert
_{H^{k+1}}^{2} +\sum_{i=1}^{ne}\left\Vert \pd{v_i}{t}\right\Vert
_{H^{k+1}}^{2} \right) h^{2(k+1)}.\\
\end{equation*}
Hence
\begin{equation*}
\left\Vert \theta_i \right\Vert\leq Ch^{k+1},\quad i=1,\dots ne
\end{equation*}
and adding the estimate of $\rho_i$, we obtain the desired result.
\end{proof}

\section{Discrete problem}

Let $\delta >0$ and consider the partition $]0,T]=\overset{n_{i}-1}{\underset%
{j=1}{\cup }}]t_{j-1},t_{j}]=\overset{n_{i}-1}{\underset{j=1}{\cup }}I_{j},$
$\delta =t_{j}-t_{j-1}$ and $int(I_{j})\cap int(I_{i})=\emptyset $. The time
discretization is made utilizing the Crank-Nicolson method. Let $\mathbf{V}%
^{(n)}(y)$ be the approximation of $\mathbf{v}(y,t_{n})$, in the space $%
(S_{h}^{k})^{ne}$. This method evaluates the equation at the points $t_{n-1/2}=\frac{t_n+t_{n-1}}2$, $n=1,\dots, ni$, and uses the approximations
\begin{equation*}
\mathbf{V}(y,t_{n-1/2})\approx\frac{\mathbf{V}^{(n)}(y)+\mathbf{V}^{(n-1)}(y)}{2}=\hat{\mathbf{V}}^{(n)}(y)
\end{equation*}%
and
\begin{equation*}
\frac{\partial \mathbf{V}}{\partial t}(y,t_{n-1/2})\approx \frac{\mathbf{V}^{(n)}(y)-\mathbf{V}
^{(n-1)}(y)}{\delta }=\overline{\partial }\mathbf{V}^{(n)}(y).
\end{equation*}
Then we have the problem of finding $\mathbf{V}^{(n)}\in (S_{h}^{k})^{ne}$
such that it is zero on the boundary of $\Omega $, satisfies $
V_{i}^{(0)}=I_{h}(v_{i0})$, $i=1,\dots ,ne$, and
\begin{equation*}
\int_{0}^{1}\overline{\partial }V_{i}^{(n)}W_{i}\ dy+a_{i}(l(\hat{V}
_{1}^{(n)}),\dots ,l(\hat{V}_{ne}^{(n)}))b_{2}^{(n-1/2)}\int_{0}^{1}\frac{
\partial \hat{V}_{i}^{(n)}}{\partial y}\frac{\partial W_{i}}{\partial y}\ dy
\end{equation*}
\begin{equation}
-\int_{0}^{1}b_{1}^{(n-1/2)}\frac{\partial \hat{V}_{i}^{(n)}}{\partial y}
W_{i}\ dy=\int_{0}^{1}g_{i}^{(n-1/2)}W_{i}\ dy,  \label{C_N}
\end{equation}%
with $f^{(n-1/2)}=f(y,t_{n-1/2})$.\newline
System (\ref{C_N}) is a non linear algebraic system due to the presence of
\linebreak $a_{i}(l(\hat{V}_{1}^{(n)}),\dots ,l(\hat{V}_{ne}^{(n)}))$.
Obtaining the solution of (\ref{C_N}) implies the use of an iterative method
in each time step. We could apply Newton's method, the fixed point method or
some secant method, but it would be very time consuming. In order to avoid
this, we choose a linearization method and, as suggested in \cite{Tho06}, we
substitute $\hat{V}_{i}^{(n)}$ with $\overline{V}_{i}^{(n)}=\frac{3}{2}%
V_{i}^{(n-1)}-\frac{1}{2}V_{i}^{(n-2)}$ in the diffusion coefficient. So the
totally discrete problem, in this case, will be to calculate the functions $%
\mathbf{V}^{(n)}$, $n\geq 2$, belonging to $(S_{h}^{k})^{ne}$, which are
zero on the boundary of $\Omega $ and satisfy
\begin{equation*}
\int_{0}^{1}\overline{\partial }V_{i}^{(n)}W_{i}\ dy+a_{i}(l(\overline{V}
_{1}^{(n)}),\dots ,l(\overline{V}_{ne}^{(n)}))b_{2}^{(n-1/2)}\int_{0}^{1}\frac{
\partial \hat{V}_{i}^{(n)}}{\partial y}\frac{\partial W_{i}}{\partial y}\ dy
\end{equation*}
\begin{equation}
-\int_{0}^{1}b_{1}^{(n-1/2)}\frac{\partial \hat{V}_{i}^{(n)}}{\partial y}
W_{i}\ dy=\int_{0}^{1}g_{i}^{(n-1/2)}W_{i}\ dy,\quad n\geq 2,\quad i=1,\dots
,ne.  \label{descn}
\end{equation}%
In this way, we have a linear multistep method which requires two initial
estimates $\mathbf{V}^{(0)}$ and $\mathbf{V}^{(1)}$. The estimate $\mathbf{V}
^{(0)}$ is obtained by the initial condition as $V_{i}^{(0)}=I_{h}(v_{i0})$.
In order to calculate $\mathbf{V}^{(1)}$ with the same accuracy, we follow
\cite{Tho06} and we use the following predictor-corrector scheme.
\begin{equation*}
\int_{0}^{1}\frac{V_{i}^{(1,0)}-V_{i}^{(0)}}{\delta }W_{i}\
dy+a_{i}(l(V_{1}^{(0)}),\dots ,l(V_{ne}^{(0)}))b_{2}^{(1/2)}
\end{equation*}
\begin{equation*}
\times\int_{0}^{1}
\frac{\partial }{\partial y}\left( \frac{V_{i}^{(1,0)}+V_{i}^{(0)}}{2}%
\right) \frac{\partial W_{i}}{\partial y}\ dy-\int_{0}^{1}b_{1}^{(1/2)}\frac{\partial }{\partial y}\left( \frac{%
V_{i}^{(1,0)}+V_{i}^{(0)}}{2}\right) W_{i}\ dy
\end{equation*}
\begin{equation}
=\int_{0}^{1}g_{i}^{(1/2)}W_{i}\ dy,\quad i=1,\dots ,ne.  \label{desc10}
\end{equation}
\begin{equation*}
\int_{0}^{1}\overline{\partial }V_{i}^{(1)}W_{i}\ dy+a_{i}\left( l\left(
\frac{V_{1}^{(1,0)}+V_{1}^{(0)}}{2}\right) ,\dots ,l\left( \frac{
V_{ne}^{(1,0)}+V_{ne}^{(0)}}{2}\right) \right) b_{2}^{(1/2)}
\end{equation*}
\begin{equation}
\times\int_{0}^{1}\frac{\partial \hat{V}_{i}^{(1)}}{\partial y}\frac{\partial W_{i}}{\partial y%
}\ dy-\int_{0}^{1}b_{1}^{(1/2)}\frac{\partial \hat{V}_{i}^{(1)}}{\partial y}%
W_{i}\ dy=\int_{0}^{1}g_{i}^{(1/2)}W_{i}\ dy,\quad i=1,\dots ,ne.
\label{desc1}
\end{equation}

\begin{theorem}
\label{conv_d} If $\mathbf{v}$ is the solution of equation (\ref{prob1}) and
$\mathbf{V}^{(n)}$ is the solution of (\ref{descn})-(\ref{desc1}), then
\begin{equation*}
\Vert V_{i}^{(n)}(y)-v_{i}(y,t_{n})\Vert \leq C(h^{k+1}+\delta ^{2}),\quad
n=1,\dots ,nt,\quad i=1,\dots ,ne,
\end{equation*}%
where $C$ does not depend on $h$, $k$ or $\delta $.
\end{theorem}

\begin{proof}
First we will determine the estimate for $n=1$.
Let $\theta_i^{(1,0)}=V_i^{(1,0)}-\tilde v_i^{(1)}$, $\hat
\theta_i^{(1,0)}=\frac{\theta_i^{(1,0)}+\theta_i^{(0)}}{2}$ and
$\overline{\partial}\theta_i^{(1,0)}=\frac{\theta_i^{(1,0)}-\theta_i^{(0)}}{\delta}$.
We have that\\

\noindent$\displaystyle \int_0^1\overline{\partial}\theta_i^{(1,0)} W_i\ dy+a_i(l(V_1^{(0)}),\dots,l(V_{ne}^{(0)}))b_2^{(1/2)}\int_0^1\frac{\partial\hat \theta_i^{(1,0)}}{\partial y}\frac{\partial W_i}{\partial y}\ dy$\\

\noindent$\displaystyle\phantom{=}-\int_0^1b_1^{(1/2)}\frac{\partial\hat \theta_i^{(1,0)}}{\partial y} W_i\ dy$\\

\noindent$\displaystyle=\int_0^1\overline{\partial}V_i^{(1,0)} W_i\ dy+a_i(l(V_1^{(0)}),\dots,l(V_{ne}^{(0)}))b_2^{(1/2)}\int_0^1\frac{\partial\hat V_i^{(1,0)}}{\partial y}\frac{\partial W_i}{\partial y}\ dy$\\

\noindent$\displaystyle\phantom{=}-\int_0^1b_1^{(1/2)}\frac{\partial\hat V_i^{(1,0)}}{\partial y} W_i\ dy-
\int_0^1\overline{\partial}\tilde v_i^{(1)} W_i\ dy$\\

\noindent$\displaystyle\phantom{=}-a_i(l(V_1^{(0)}),\dots,l(V_{ne}^{(0)}))b_2^{(1/2)}\int_0^1\frac{\partial\hat{\tilde v}_i^{(1)}}{\partial y}\frac{\partial W_i}{\partial y}\ dy+\int_0^1b_1^{(1/2)}\frac{\partial\hat{\tilde v}_i^{(1)}}{\partial y} W_i\ dy$\\

\noindent$\displaystyle=\int_0^1g_i^{(1/2)}W_i\ dy-\int_0^1\left(\pd{v_i}{t}\right)^{(1/2)} W_i\ dy-a_i(l(v_1^{(1/2)}),\dots,l(v_{ne}^{(1/2)}))b_2^{(1/2)}$\\

\noindent$\displaystyle\phantom{=}\times\int_0^1\frac{\partial v_i^{(1/2)}}{\partial y}\frac{\partial W_i}{\partial y}\ dy+\int_0^1b_1^{(1/2)}\frac{\partial v_i^{(1/2)}}{\partial y} W_i\ dy+\int_0^1\left(\left(\pd{v_i}{t}\right)^{(1/2)}-\overline\partial\tilde v_i^{(1)}\right)W_i\ dy$\\

\noindent$\displaystyle\phantom{=}+b_2^{(1/2)}\int_0^1\left(a_i(l(v_1^{(1/2)}),\dots,l(v_{ne}^{(1/2)}))\pd{v_1^{(1/2)}}{y}\right.$\\

\noindent$\displaystyle\phantom{=}\left.-a_i(l(V_1^{(0)}),\dots,l(V_{ne}^{(0)}))\pd{\hat{v}_i^{(1)}}{y}\right)\pd{W_i}{y}\ dy+\int_0^1b_1^{(1/2)}\left(\pd{\hat{\tilde v}_i^{(1)}}{y}-\pd{v_i^{(1/2)}}{y}\right)W_i\ dy$\\

\noindent$\displaystyle=\int_0^1\left(\left(\pd{v_i}{t}\right)^{(1/2)}-\overline\partial\tilde v_i^{(1)}\right)W_i\ dy+a_i(l(v_1^{(1/2)}),\dots,l(v_{ne}^{(1/2)}))b_2^{(1/2)}$\\

\noindent$\displaystyle\phantom{=}\times\int_0^1 \left(\pd{v_i^{(1/2)}}{y}-\pd{\hat{v}_i^{(1)}}{y}\right)\pd{W_i}{y}\ dy+\left(a_i(l(v_1^{(1/2)}),\dots,l(v_{ne}^{(1/2)}))\right.$\\

\noindent$\displaystyle\phantom{=}\left.-a_i(l(V_1^{(0)}),\dots,l(V_{ne}^{(0)}))\right)b_2^{(1/2)}\int_0^1 \pd{\hat{v}_i^{(1)}}{y}\pd{W_i}{y}\ dy$\\

\noindent$\displaystyle\phantom{=}+\int_0^1b_1^{(1/2)}\left(\pd{\hat{\tilde v}_i^{(1)}}{y}-\pd{v_i^{(1/2)}}{y}\right)W_i\ dy.$\\

\noindent Setting  $W_i=\hat\theta_i^{(1,0)}$, we arrive at\\

\noindent$\displaystyle\int_0^1\overline{\partial}\theta_i^{(1,0)} \hat\theta_i^{(1,0)}\ dy+a_i(l(V_1^{(0)}),\dots,l(V_{ne}^{(0)}))b_2^{(1/2)}\int_0^1\left(\frac{\partial\hat \theta_i^{(1,0)}}{\partial y}\right)^2 dy$\\

\noindent$\displaystyle-\int_0^1b_1^{(1/2)}\frac{\partial\hat \theta_i^{(1,0)}}{\partial y} \hat\theta_i^{(1,0)}\ dy$\\

\noindent$\displaystyle=\int_0^1\left(\left(\pd{v_i}{t}\right)^{(1/2)}-\overline\partial\tilde v_i^{(1)}\right)\hat\theta_i^{(1,0)}\ dy$\\

\noindent$\displaystyle\phantom{=}+a_i(l(v_1^{(1/2)}),\dots,l(v_{ne}^{(1/2)}))b_2^{(1/2)}\int_0^1 \left(\pd{v_i^{(1/2)}}{y}-\pd{\hat{v}_i^{(1)}}{y}\right)\pd{\hat\theta_i^{(1,0)}}{y}\ dy$\\

\noindent$\displaystyle\phantom{=}+\left(a_i(l(v_1^{(1/2)}),\dots,l(v_{ne}^{(1/2)}))-a_i(l(V_1^{(0)}),\dots,l(V_{ne}^{(0)}))\right)b_2^{(1/2)}$\\

\noindent$\displaystyle\phantom{=}\times\int_0^1 \pd{\hat{v}_i^{(1)}}{y}\pd{\hat\theta_i^{(1,0)}}{y}\ dy+\int_0^1b_1^{(1/2)}\left(\pd{\hat{\tilde v}_i^{(1)}}{y}-\pd{v_i^{(1/2)}}{y}\right)\hat\theta_i^{(1,0)}\ dy.$\\

\noindent Applying integration by parts and hypothesis $H_1$ and
$H_2$, it follows that
$$-\int_0^1b_1^{(1/2)}\frac{\partial\hat \theta_i^{(1,0)}}{\partial y} \hat\theta_i^{(1,0)}\ dy=-\int_0^1b_1^{(1/2)}\frac12\frac{\partial(\hat\theta_i^{(1,0)})^2}{\partial y}\ dy$$
$$=\frac{(\gamma')^{(1/2)}}{2\gamma^{(1/2)}}\int_0^1(\hat\theta_i^{(1,0)})^2\ dy\geq 0$$
and
$$\int_0^1b_1^{(1/2)}\left(\pd{\hat{\tilde v}_i^{(1)}}{y}-\pd{v_i^{(1/2)}}{y}\right)\hat\theta_i^{(1,0)}\ dy$$
$$=-\frac{(\gamma')^{(1/2)}}{\gamma^{(1/2)}}\int_0^1 \left(\hat{\tilde v}_i^{(1)}-v_i^{(1/2)}\right)\hat\theta_i^{(1,0)}\ dy-
\int_0^1 b_1^{(1/2)}(\hat{\tilde v}_i^{(1)}-v_i^{(1/2)})\frac{\partial\hat\theta_i^{(1,0)}}{\partial y}\ dy.$$
Then, by the Poincaré and Hölder inequalities, we can conclude that\\

$\displaystyle
\frac12\overline{\partial}\|\theta_i^{(1,0)}\|^2+\frac{m_a}{\gamma_0^2}\left\|\frac{\partial\hat\theta_i^{(1,0)}}{\partial y}\right\|^2\leq$
$$\leq C \left( \left\|\left(\pd{v_i}{t}\right)^{(1/2)}-\overline\partial\tilde v_i^{(1)}\right\|+\left\|\pd{v_i^{(1/2)}}{y}-\pd{\hat{v}_i^{(1)}}{y}\right\|+\sum_{j=1}^{ne}\|v_{j}^{(1/2)}-V_{j}^{(0)}\|\right.$$
$$+\|\hat{\tilde v}_i^{(1)}-v_i^{(1/2)}\|\Bigg)\left\|\frac{\partial\hat\theta_i^{(1,0)}}{\partial y}\right\|.$$
Using Cauchy's inequality, we have that
\begin{eqnarray*}
\overline{\partial}\|\theta_i^{(1,0)}\|^2&\leq& C \left( \left\|\left(\pd{v_i}{t}\right)^{(1/2)}-\overline\partial\tilde v_i^{(1)}\right\|+\left\|\pd{v_i^{(1/2)}}{y}-\pd{\hat{v}_i^{(1)}}{y}\right\|\right.\\ &&+\sum_{j=1}^{ne}\|v_{j}^{(1/2)}-V_{j}^{(0)}\|+\|\hat{\tilde v}_i^{(1)}-v_i^{(1/2)}\|\Bigg).
\end{eqnarray*}
The following estimates are true for every $i\in\{1,\dots,ne\}$,
\begin{eqnarray*}
\left\|\left(\pd{v_i}{t}\right)^{(1/2)}-\overline\partial\tilde v_i^{(1)}\right\|&\leq&\left\|\left(\pd{v_i}{t}\right)^{(1/2)}-\overline\partial v_i^{(1)}\right\|+\|\overline\partial v_i^{(1)}-\overline\partial\tilde v_i^{(1)}\|\\
&\leq &C\delta^2+C h^{k+1},\\
\left\|\pd{v_i^{(1/2)}}{y}-\pd{\hat{v}_i^{(1)}}{y}\right\|&\leq& C\delta\int_{t_0}^{t_1}\left\|\frac{\partial^3 v_{i}}{\partial y\partial t^2}\right\|\ dt\leq C\delta^2,\\
\|v_{i}^{(1/2)}-V_{i}^{(0)}\|&\leq&\|v_{i}^{(1/2)}-v_{i}^{(0)}\|+\|v_{i}^{(0)}-V_{i}^{(0)}\|\leq C\delta+C h^{k+1},\\
\end{eqnarray*}
and
$$\|\hat{\tilde v}_i^{(1)}-v_i^{(1/2)}\|\leq  \|\hat{\tilde v}_i^{(1)}-\hat{\tilde v}_i^{(1/2)}\|+\|\hat{\tilde v}_i^{(1/2)}-v_i^{(1/2)}\|\leq C\delta^2+C h^{k+1}.$$
Hence
$$\overline{\partial}\|\theta_i^{(1,0)}\|^2\leq C (h^{k+1}+\delta)^2,$$
and, we have the estimate
$$\|\theta_i^{(1,0)}\|^2\leq\|\theta_i^{(0)}\|^2+ C\delta (h^{k+1}+\delta)^2\leq C(h^{2(k+1)}+\delta^3),\quad i=1,\dots,ne.$$
Repeating this process for equation (\ref{desc1}), we arrive at\\

$\displaystyle
\frac12\overline{\partial}\|\theta_i^{(1)}\|^2+\frac{m_a}{\gamma_0^2}\left\|\frac{\partial\hat\theta_i^{(1)}}{\partial y}\right\|^2$
\begin{eqnarray*}
&\leq& C \left( \left\|\left(\pd{v_i}{t}\right)^{(1/2)}-\overline\partial\tilde v_i^{(1)}\right\|+\left\|\pd{v_i^{(1/2)}}{y}-\pd{\hat{v}_i^{(1)}}{y}\right\|\right.\\
&&\left.+\sum_{j=1}^{ne}\left\|v_{j}^{(1/2)}-\frac{V_{j}^{(1,0)}-V_{j}^{(0)}}{2}\right\|+\|\hat{\tilde v}_i^{(1)}-v_i^{(1/2)}\|\right)\left\|\frac{\partial\hat\theta_i^{(1)}}{\partial y}\right\|.\\
\end{eqnarray*}
In this case, we use the estimate
\begin{eqnarray*}
\left\|v_{i}^{(1/2)}-\frac{V_{i}^{(1,0)}-V_{i}^{(0)}}{2}\right\|&\leq& \|v_{i}^{(1/2)}-\hat{\tilde v}_i^{(1)}\|+\|\hat{\tilde v}_i^{(1)}-\frac{V_{i}^{(1,0)}-V_{i}^{(0)}}{2}\|\\
&\leq& \|v_{i}^{(1/2)}-\hat{\tilde v}_i^{(1)}\|+\frac12\|\theta_i^{(1,0)}\|+\frac12\|\theta_i^{(0)}\|\\
&\leq& C(h^{k+1}+\delta^{2})+C h^{k+1}+C(h^{k+1}+\delta^{\frac32})\\
&\leq& C(h^{k+1}+\delta^{\frac32}),
\end{eqnarray*}
and then, by Cauchy's inequality, we conclude that
$$\overline{\partial}\|\theta_i^{(1)}\|^2\leq C (h^{2(k+1)}+\delta^3),$$
whence
$$\|\theta_i^{(1)}\|^2\leq\|\theta_i^{(0)}\|^2+ C\delta (h^{2(k+1)}+\delta^3)\leq C(h^{2(k+1)}+\delta^4).$$
To conclude the proof, we obtain the result for $n\geq2$, applying the same process to equation (\ref{descn}). In this way, we obtain\\

$\displaystyle
\frac12\overline{\partial}\|\theta_i^{(n)}\|^2+\frac{m_a}{\gamma_0^2}\left\|\frac{\partial\hat\theta_i^{(n)}}{\partial y}\right\|^2$
\begin{eqnarray*}
&\leq& C \left( \left\|\left(\pd{v_i}{t}\right)^{(n-1/2)}-\overline\partial\tilde v_i^{(n)}\right\|+\left\|\pd{v_i^{(n-1/2)}}{y}-\pd{\hat{v}_i^{(n)}}{y}\right\|\right.\\
&&\left.+\sum_{j=1}^{ne}\left\|v_{j}^{(n-1/2)}-\bar{V}_j^{(n)}\right\|+\|\hat{\tilde v}_i^{(n)}-v_i^{(n-1/2)}\|\right)\left\|\frac{\partial\hat\theta_i^{(n)}}{\partial y}\right\|.
\end{eqnarray*}
Now, we need the estimate
\begin{eqnarray*}
\left\|v_{i}^{(n-1/2)}-\bar{V}_i^{(n)}\right\|&\leq& \|v_{i}^{(n-1/2)}-\bar{v}_i^{(n)}\|+\|\bar{v}_i^{(n)}-\bar{V}_i^{(n)}\|\\
&\leq& \|v_{i}^{(n-1/2)}-\bar{v}_i^{(n)}\|+\|\overline\rho_i^{(n)}\|+\|\overline\theta_i^{(n)}\|\\
&\leq& C\delta^2+C h^{k+1}+C(\|\theta_{n-1}\|+\|\theta_{n-2}\|)
\end{eqnarray*}
to prove that
$$\overline{\partial}\|\theta_i^{(n)}\|^2\leq C\sum_{j=1}^{ne}\|\theta_j^{(n-1)}\|^2+C\sum_{j=1}^{ne}\|\theta_j^{(n-2)}\|^2+C(h^{(k+1)}+\delta^2)^2,\quad i=1,\dots,ne.$$
Summing up for all $i$, it follows that
$$\overline{\partial}\sum_{i=1}^{ne}\|\theta_i^{(n)}\|^2\leq C\sum_{j=1}^{ne}\|\theta_j^{(n-1)}\|^2+C\sum_{j=1}^{ne}\|\theta_j^{(n-2)}\|^2+C(h^{(k+1)}+\delta^2)^2.$$
Iterating, we obtain
$$\sum_{i=1}^{ne}\|\theta_i^{(n)}\|^2\leq(1+C\delta)\sum_{i=1}^{ne}\|\theta_i^{(n-1)}\|^2+ C\delta\sum_{i=1}^{ne}\|\theta_i^{(n-2)}\|^2 +C\delta (h^{k+1}+\delta^2)^2$$
$$\leq C\sum_{i=1}^{ne}\|\theta_i^{(1)}\|^2+ C\sum_{i=1}^{ne}\delta\|\theta_i^{(0)}\|^2+ C\delta (h^{k+1}+\delta^2)^2$$
and recalling the estimates for $\|\theta_i^{(0)}\|$, $\|\theta_i^{(1)}\|$ and
$\|\rho_i^{(n)}\|$, the proof is complete.
\end{proof}


\section{Examples}

The final step is to implement this method using a programming language. To
perform this task, we choose the Matlab environment.

In this section, we present some examples to illustrate the applicability
and robustness of this method, comparing the results with the theoretical
results proved and with the results presented in \cite{RACFppb}.

\subsection{Example 1}

As a first example we simulate a problem with a known exact solution, which
will permit us to calculate the error and confirm numerically the
theoretical convergence rates. Let us consider problem (\ref{prob0}) with
two equations in $Q_{t}$ and $T=3$. The diffusion
coeficientes are
\begin{equation*}
a_1(r,s)=2-\frac{1}{1+r^2}+\frac{1}{1+s^2},\quad a_2(r,s)=3+\frac{2}{1+r^2}-%
\frac{1}{1+s^2},
\end{equation*}
the movement of the boundaries is given by the functions
\begin{equation*}
\alpha(t)=-\frac{t}{1+t},\quad \beta(t)=1+\frac{2t}{1+t},
\end{equation*}
the functions $f_1(x,t)$, $f_2(x,t)$, $u_{10}(x,t)$ and $u_{20}(x,t)$ are
chosen such that
\begin{equation*}
u_1(x,t)=\frac{1}{t+1}\left(\frac{611}{70}z-\frac{10513}{210}z^2+\frac{646}{7%
}z^3-\frac{1070}{21}z^4\right)
\end{equation*}
and
\begin{equation*}
u_2(x,t)=e^{-t}\left(\frac{2047}{140}z-\frac{27701}{420}z^2+\frac{691}{7}z^3-%
\frac{995}{21}z^4\right)
\end{equation*}
with
\begin{equation*}
z=\frac{(2t+1)(x+tx+t)}{5t^2+5t+1}
\end{equation*}
are the exact solutions.

\begin{figure}[!htb]
\center
\includegraphics[width=\textwidth]{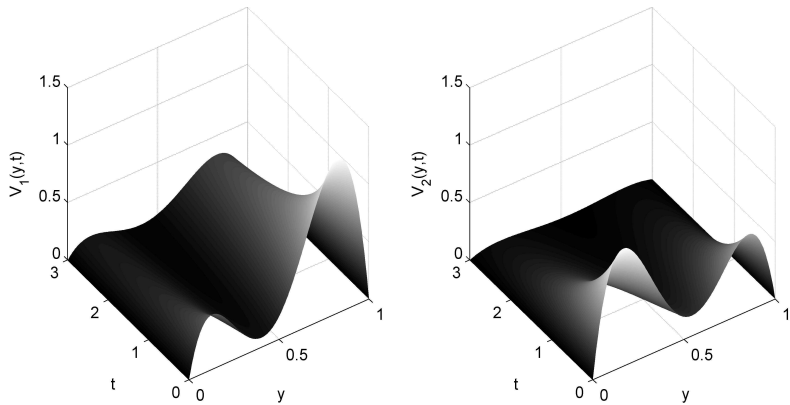}
\caption{Evolution in time of the approximated solution in the fixed
boundary problem for $v_{1}$ (left) and $v_{2}$ ( right).}
\label{ex1_sol_f}
\end{figure}
The picture on the left in Figure \ref{ex1_sol_f}
illustrates the evolution in time of the obtained solution for $v_1$ in the
fixed boundary problem, and the picture on the right illustrates the
evolution in time of the obtained solution for $v_2$. This solution was
calculated with approximations of degree two and $h=\delta=10^{-2}$.

\begin{figure}[!htb]
\center
\includegraphics[width=\textwidth]{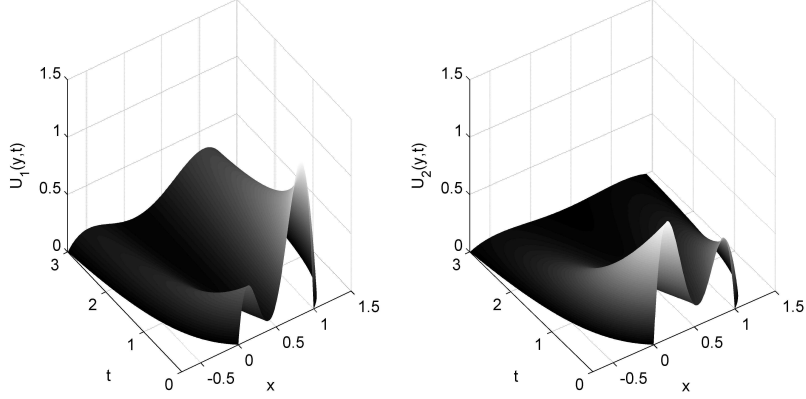}
\caption{Evolution in time of the approximated solution in the moving
boundary problem for $u_{1}$ (left) and $u_{2}$ ( right).}
\label{ex1_sol_m}
\end{figure}
The pictures in Figure \ref{ex1_sol_m} represent the obtained solutions in
the moving boundary domain, after applying the inverse transformation $\tau
^{-1}(y,t)$. In this case $u$ and $v$ could represent the density of two populations
of bacteria. We observe that, initially, each population is concentrated mainly in two
regions and, as time increases, the two populations decrease and spread out
in the domain, as expected.

In order to analyze the convergence rates, this problem was simulated with
different combinations of $k$, $h$ and $\delta $ and the error results are
represented in Figure \ref{ex1_O_h}.

\begin{figure}[tbh]
\centering
\includegraphics[width=0.32\textwidth]{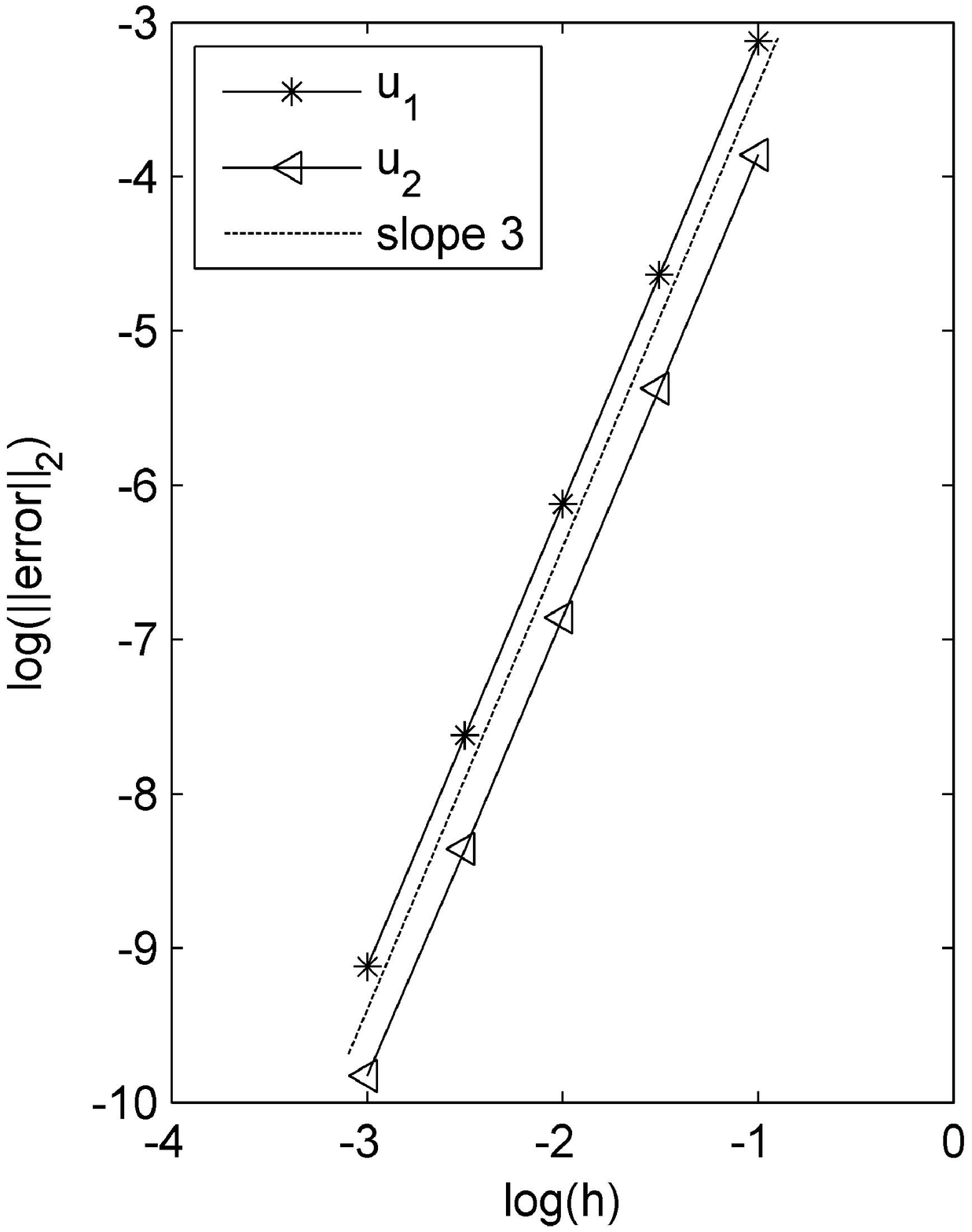} %
\includegraphics[width=0.32\textwidth]{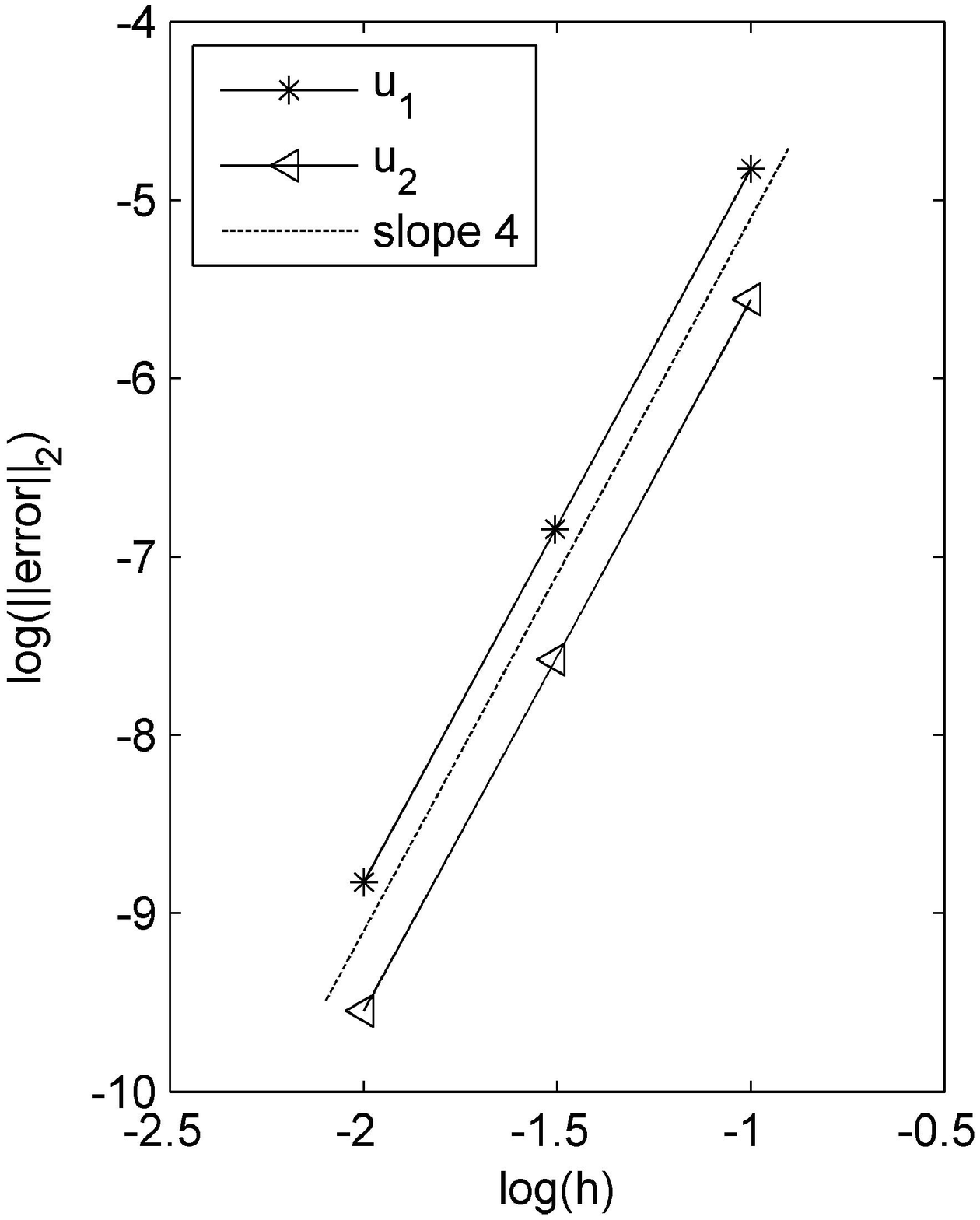} %
\includegraphics[width=0.32\textwidth]{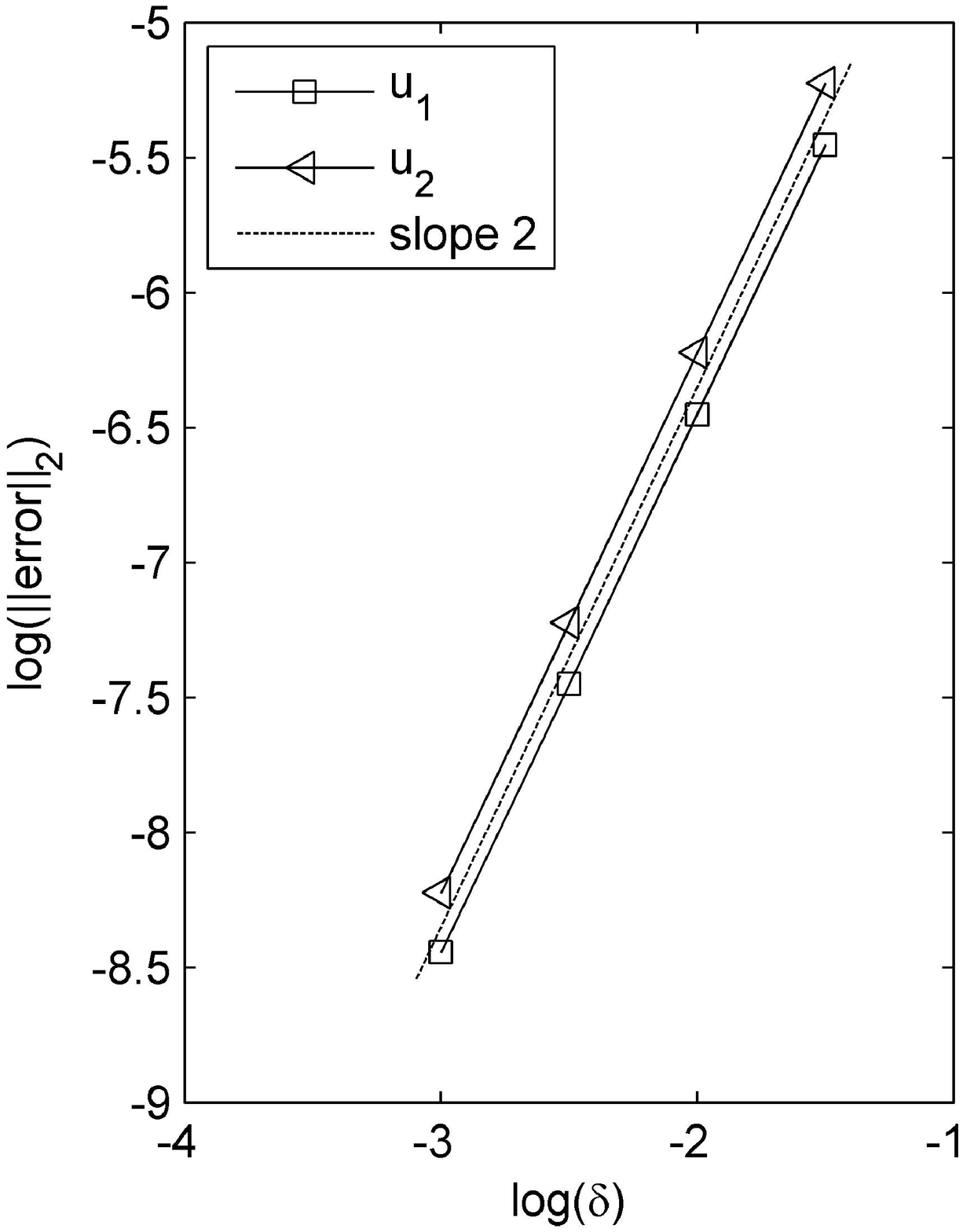}
\caption{Study of the convergence for $h$ with approximations of degree 2
(left) and 3 (center), and for $\protect\delta $ (right).}
\label{ex1_O_h}
\end{figure}
The error was calculated in $t=T$ and using the $L_{2}(\alpha (T),\beta (T))$%
-norm in the space variable. In the  picture on the left  the logarithms of
the errors versus the logarithm of $h$ for the simulations done with $\delta
=10^{-4}$ and approximations of degree $2$, are represented.  The errors
versus the logarithm of $h$ for the simulations done with $\delta =10^{-4}$
and approximations of degree $3$ are represented in the picture in the
center. The logarithms of the errors versus the logarithm of $\delta $ for
the simulations done with $h=10^{-3}$ and approximations of degree $2$, are
represented in the picture on the right. As expected, the pictures are in
accordance with the orders of convergence for $h$ and $\delta $, as was
proved in Theorem \ref{conv_d}. In Table \ref{tabela1} we compare the error
of the present method with the error of the moving finite element method
presented in \cite{RACFppb}. Both simulations were done with approximations
of degree five and four finite elements. We used $\delta =10^{-4}$ for the
present method and $10^{-10}$ for the integrator's error tolerance in the
moving finite element method.

\begin{table}[!htb]
\begin{tabular}{l|c|c|c|c|}
& \multicolumn{2}{|c|}{$\displaystyle\max_{j=1,\dots,np}%
\{|u_1(P_j,t_i)-U_1^{(i)}(P_j)|\}$} & \multicolumn{2}{|c|}{$\displaystyle%
\max_{j=1,\dots,np}\{|u_2(P_j,t_i)-U_2^{(i)}(P_j)|\}$} \\[0.3cm]
$t_i$ & \hspace{0.6cm}MFEM\cite{RACFppb}\hspace{0.6cm} & present & \hspace{%
0.6cm}MFEM\cite{RACFppb}\hspace{0.6cm} & present \\ \hline
$0.001$ & 7.3025e-08 & 2.6502e-10 & 4.2464e-08 & 6.3606e-10 \\
$0.005$ & 8.9490e-08 & 1.0338e-09 & 5.2037e-08 & 1.4457e-09 \\
$0.01$ & 2.7945e-08 & 1.4645e-09 & 1.6249e-08 & 1.8010e-09 \\
$0.02$ & 1.3320e-08 & 1.8424e-09 & 7.7437e-09 & 2.0228e-09 \\
$0.05$ & 7.2664e-08 & 2.0530e-09 & 4.2203e-08 & 2.1597e-09 \\
$0.5$ & 1.9044e-08 & 1.0564e-09 & 1.0743e-08 & 1.0907e-09 \\
$1$ & 2.1230e-08 & 5.0614e-10 & 9.3304e-09 & 5.5859e-10
\end{tabular}%
\caption{Comparison of the present method with the moving finite element
method in \protect\cite{RACFppb}}
\label{tabela1}
\end{table}
\subsection{Example 2}

As a second example, we choose to simulate the second example presented in
\cite{RACFppb}. This will permit us to compare the present method with an
adaptive one. Consider problem (\ref{prob0}) with $ne=2$ and $Q_{t}$ defined by
\begin{equation*}
\alpha (t)=\sqrt{2/3}-\sqrt[3]{t+(2/3)^{3/2}},\quad \beta (t)=1-\alpha
(t),\quad 0\leq t\leq 1.
\end{equation*}%
The diffusion coefficients are
\begin{equation*}
a_{1}(r,s)=2-\frac{1}{1+s^{2}},\quad a_{2}(r,s)=e^{-r^{2}},
\end{equation*}%
and the reaction forces are
\begin{equation*}
f_{1}(x,t)=\frac{0.1x}{(1+t)^{4}},\quad f_{2}(x,t)=\frac{e^{-x^{2}}}{%
(1+t)^{6}}.
\end{equation*}%
The initial conditions $u_{10}$ and $u_{20}$ are the natural spline
functions of degree three that interpolate the points
$\{(0,0),(0.2,1),(0.5,0.5),(1,0)\}$ and
$\{(0,0),\linebreak (0.6,0.65),(0.8,1),(1,0)\},$
respectively. The approximate solutions were obtained with four finite
elements ($h=0.25$), $\delta =10^{-3}$ and $k=4$. The obtained solutions in
the fixed domain are plotted in Figure $4$.

\begin{figure}[!htb]
\center
\includegraphics[width=\textwidth]{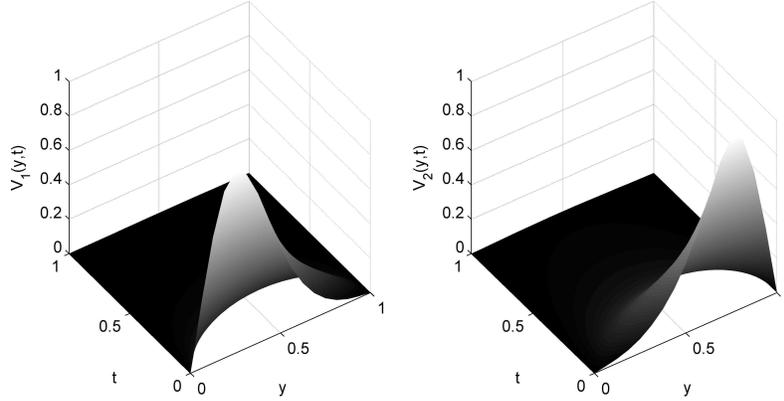}
\caption{Evolution in time of the approximated solution in the fixed
boundary problem for $v_{1}$ (left) and $v_{2}$ ( right).}
\label{ex2_sol_f}
\end{figure}
The pictures in Figure \ref{ex2_sol_m} represent the obtained solutions in
the moving boundary domain, after applying the inverse transformation $\tau
^{-1}(y,t)$. In this example, initially, each population occupies mainly one
region opposite from the other population. As the time increases the two
populations expands to all the domain and decreases very quickly.

\begin{figure}[!htb]
\center
\includegraphics[width=\textwidth]{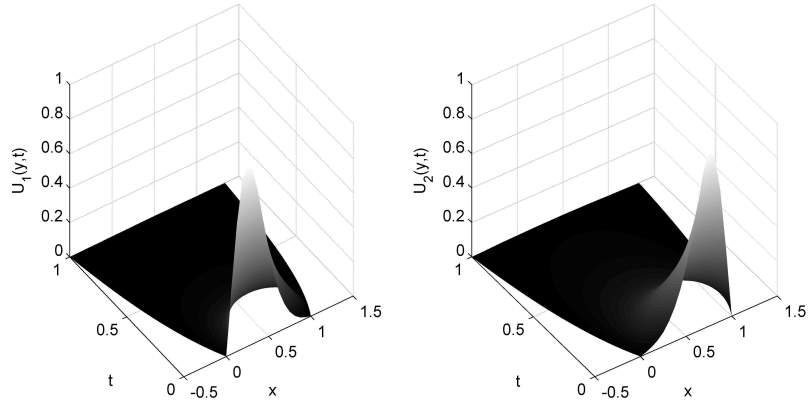}
\caption{Evolution in time of the approximated solution in the moving
boundary problem for $u_{1}$ (left) and $u_{2}$ ( right).}
\label{ex2_sol_m}
\end{figure}
The pictures are similar to those in \cite{RACFppb} and the numerical
comparisons between the two methods show that the methods are similar.
However, due to the fact that in \cite{RACFppb} an adaptive mesh was used,
initially the difference between the methods is greater in the areas where
the solution has a higher slope, but this difference become less significant
as time grows.

\section{Conclusions}

We proved optimal rates of convergence for a linearized
Crank-Nicolson-Galerkin finite element method with piecewise polynomial of
arbitrary degree basis functions in space when applied to a system of
nonlocal parabolic equations. Some numerical experiments were presented,
considering different functions $a$, $f$, $\alpha $ and $\beta $. The
numerical results are in accordance with the theoretical results and are
similar in accuracy to results obtained by other methods.

\section*{Acknowledgements}

This work was partially supported by the research projects:\newline
PEst-OE/MAT/UI0212/2011, financed by FEDER
through the - Programa O\-pe\-ra\-ci\-o\-nal Factores de Competitividade, FCT -
Fundação para a Ciência e a Tecnologia and CAPES - Brazil, Grant BEX 2478-12-9.

\bibliographystyle{plain}
\bibliography{ADFR}

\end{document}